\documentclass[11pt]{article}

\usepackage{MyMacros} 

\newcommand{\muHR}{\mu_{\rm HR}}
\newcommand{\FgHR}{\cF^{\rm gen HR}}
\newcommand{\fgHR}{f^{\rm gen HR}}

\title{On pairwise interaction multivariate Pareto models}
\author{Michaël Lalancette\thanks{\texttt{michael.lalancette@tum.de}}}
\affil{Department of Mathematics, Technical University of Munich, Germany}
\date{\today}

\begin{document}

\maketitle

\begin{abstract}
	
	The rich class of multivariate Pareto distributions forms the basis of recently introduced extremal graphical models. However, most existing literature on the topic is focused on the popular parametric family of H\"usler--Reiss distributions. It is shown that the H\"usler--Reiss family is in fact the only continuous multivariate Pareto model that exhibits the structure of a pairwise interaction model, justifying its use in many high-dimensional problems. Along the way, useful insight is obtained concerning a certain class of distributions considered in \cite{LO23}, a result of independent interest.
	
\end{abstract}

\section{Introduction and main result}

Multivariate Pareto distributions play a central role in tail dependence modeling and inference as the only non-degenerate limit laws that can arise from multivariate threshold exceedances. They are defined as the class of possible non-degenerate weak limits of $u^{-1} \bX \given \|\bX\|_\infty > u$, as $u \to \infty$, for random vectors $\bX$ with unit Pareto margins. As such, they are usually considered to perfectly describe the possible tail dependence strucures of multivariate data. Originally introduced in \cite{roo2006}, they form the basis of multivariate peaks-over-threshold inference \citep{RSW18POT,KRSW19}. Other than a constraint on their support and a certain marginal standardization arising from their definition, multivariate Pareto distributions essentially comprise all homogeneous distributions. In the absolutely continuous case, which is the focus of this note, they can be exactly defined as follows \citep[cf.][]{EV22}.

\begin{defin}
\label{defin:MP}
	An absolutely continuous random vector $\bY := (Y_1, \dots, Y_d)$ with density $f$ has a \emph{multivariate Pareto} distribution if:
	\begin{itemize}
			\item[(MP1)] it is supported on $\cL := [0, \infty)^d \setminus [0, 1]^d$, i.e., $f(\by) = 0$ for $\by \notin \cL$,
			\item[(MP2)] its density $f$ is $(d+1)$-homogeneous, i.e., $f(t \by) = t^{-(d+1)} f(\by)$ for $\by \in \cL$ and $t \geq 1$, and
			\item[(MP3)] the probabilities $P(Y_k > 1)$, $k \in {1, \dots, d}$, are equal to each other.
	\end{itemize}
\end{defin}

The class of multivariate Pareto distributions is of course very rich. It is equivalent to the class of all extreme value copulas, or to all multivariate max-stale distributions (with fixed margins). As such, parametric subfamilies of multivariate Pareto distributions are derived from the corresponding multivariate max-stable models. The most popular such model is that associated to the family of H\"usler--Reiss distributions \cite{HR1989}, hereafter termed H\"usler--Reiss distributions themselves for convenience (rather than H\"usler--Reiss multivariate Pareto).

In the example below and in the rest of this note, let $\bone$ denote a vector each element of which is $1$, and the dimension of which will be clear from the context. For a vector $\by := (y_1, \dots, y_d)$, we define $\log\by$ as the elementwise (natural base) logarithm $(\log y_1, \dots, \log y_d)$. The space of symmetric $d \times d$ matrices is denotes by $\cS^{d \times d}$, and $\cS_{\bone}^{d \times d} \subset \cS^{d \times d}$ represents those matrices which have $\bone$ in their kernel (i.e., that have zero row and column sums). Finally, $\cS_{\bone, +}^{d \times d} \subset \cS_{\bone}^{d \times d}$ represents the matrices which, in addition, are positive semi-definite with rank $d-1$.

\begin{ex}[H\"usler--Reiss distribution]
\label{ex:HR}
	The multivariate Pareto distributed random vector $\bY$ has a \emph{H\"usler--Reiss} distribution if for a matrix $\Theta \in \cS_{\bone, +}^{d \times d}$, its density $f$ is given by
	\[
		f(\by) \propto \exp\Big\{ -\muHR(\Theta)^\top (\log\by) - (\log\by)^\top \Theta (\log\by) \Big\}, \quad \by \in \cL,
	\]
	where $\muHR(\Theta) := (1 + \tfrac{2}{d})\bone - \tfrac{1}{d}\Theta\Gamma\bone$, $\Gamma := \bone \diag(\Theta^+)^\top + \diag(\Theta^+) \bone^\top - 2\Theta^+$, and $\Theta^+$ denotes the Moore--Penrose pseudoinverse of $\Theta$. While traditionally parametrized by the \emph{variogram matrix} $\Gamma$, it has recently been suggested that the H\"usler--Reiss family can be elegantly parametrized by the \emph{H\"usler--Reiss precision matrix} $\Theta$ \citep{HES22}. By Proposition 3.4 of that paper, these two matrices are uniquely determined by each other through a homeomorphic mapping between $\cS_{\bone, +}^{d \times d}$ and the space of symmetric, strictly conditionally negative definite matrices, to which $\Gamma$ belongs. We follow \cite{HES22} and shall refer to the H\"usler--Reiss distribution with precision matrix $\Theta$.
\end{ex}

H\"usler--Reiss distributions enjoy a nice connection to recently introduced extremal graphical models \citep{EH2020}. The authors of that paper declare that two components $Y_i$ and $Y_j$ of a multivariate Pareto random vector are conditionally independent given the other variables $(Y_k)_{k \notin \{i, j\}}$ in the extremal sense if, roughly speaking, the density of $\bY$ admits the factorization
\begin{equation} \label{eq:factor}
	f(\by) = f_i(\by) f_j(\by),
\end{equation}
where $f_i$ (respectively $f_j$) does not depend on its $i$th (respectively $j$th) argument; see their Proposition 1. While this would be equivalent to the usual notion of conditional independence were $\bY$ supported on a product space, this is not the case for multivariate Pareto distributions, which are supported on $\cL$. They then define an extremal graphical model as a multivariate Pareto distribution satisfying a Markov property (with respect to this weaker notion of conditional independence) on a given undirected graph. For the relevant notions of graphical modeling, the reader is referred to \cite{EH2020} or to \cite{MDLW18}.

It is straightforward that if $\bY$ is H\"usler--Reiss distributed with precision matrix $\Theta$, then $Y_i$ and $Y_j$ are conditionally independent given the other variables (in the extremal sense of \cite{EH2020}) if and only if $\Theta_{ij} = 0$. This forms an example of a \emph{pairwise interaction model}: an exponential family of multivariate distributions where the $(i,j)$th element of a parameter matrix fully governs the dependence between the $i$th and $j$th variables.

\begin{defin}
\label{defin:PI}
	Let $q \in \N$. A (curved) exponential family of probability distributions supported on a common set $\cY \subseteq \R^d$ and indexed by a parameter space $\Omega \subseteq (\R^q)^d \times \cS^{d \times d}$ will be called a \emph{pairwise interaction model} if:
	\begin{itemize}
		\item[(PI1)] it corresponds to a family of Lebesgue densities
		\begin{equation} \label{eq:pimodel}
			\cF := \bigg\{ \by \mapsto f_{\mu, \Theta}(\by) := \frac{1}{Z(\mu, \Theta)} \exp\Big\{ -\sum_{i=1}^d \mu_i^\top S_i(y_i) - \sum_{i=1}^d \sum_{j=1}^d \Theta_{ij} T_i(y_i) T_j(y_j) \Big\}, \quad (\mu, \Theta) \in \Omega \bigg\}
		\end{equation}
		for some measurable functions $S_i: \R \to \R^q$ and measurable, non-constant functions $T_i: \R \to \R$, and
		\item[(PI2)] for every pair $(i, j)$, $i \neq j$, there exists at least one parameter $(\mu, \Theta) \in \Omega$ such that $\Theta_{ij} \neq 0$.
	\end{itemize}
\end{defin}


\begin{rem}
	\hfill
	\begin{enumerate}
		
		\item Only Lebesgue densities are considered for the purpose of the present note, but \cref{defin:PI} can be readily adapted to pairwise interaction models dominated by any base (product) measure.
		
		\item Following standard definitions of exponential families, a so-called carrying density $\prod_{i=1}^d h_i(y_i)$ could be included as a factor in the functions $f_{\mu, \Theta}$. However, since we do not require the family to be regular (in particular, the H\"usler--Reiss distributions form a curved exponential family), the carrying density can be absorbed in the marginal terms $S_i(y_i)$ by possibly adding a dimension to each $\mu_i$ and reducing the domain $\cY$. The reader is referred to standard texts such as \cite{Brown86} for notions of exponential families.
		
		\item The property (PI2) is a technical requirement for the characterization of all pairwise interaction multivariate Pareto models in \cref{lemm:F} below. It is very minor: it holds if at least one possible value of the parameter $\Theta$ has no zeros, i.e., if the model allows for the simultaneous presence of all pairwise interactions. Upon inspection of the proof of \cref{lemm:F}, it could even be further relaxed. For instance, consider the graph $G$ which has an edge between $i$ and $j$ if and only if $\Theta_{ij} \neq 0$ for some $(\mu, \Theta) \in \Omega$. The property (PI2) states that $G$ is connected, but \cref{lemm:F} holds under the mere requirement that every edge in $G$ is part of a cycle of odd length (e.g., a triangle).
		
	\end{enumerate}
\end{rem}

Pairwise interaction models are ubiquitous in dependence modeling. When the common support $\cY$ is a product space, they are particularly elegant examples of undirected graphical models, where the conditional independence graph contains the edge $(i, j)$ if and only if $\Theta_{ij} \neq 0$. Each variable then typically satisfies a generalized linear model conditionally on the other variables, with the regression coefficients being extracted from the parameter matrix. Structure learning and parameter inference on a given graph structure can be efficiently carried out via (possibly penalized) likelihood, but also regression \citep{YRAL15} or score matching \citep{LDS16} based methods. Gaussian and Gaussian copula models \citep{LLW09} as well as the continuous square root graphical model \citep{IRD16} are examples of pairwise interaction graphical models used for high-dimensional dependence modeling in Euclidean settings. \cite{KOBMK20} introduce pairwise interaction graphical models for multivariate angular data. Discrete analogues include the Ising model and more general log-linear interaction models \citep{DLS80} as well as the discrete square root graphical model \citep{IRD16}.


Even in a pairwise interaction model without a product space support (such as multivariate Pareto distributions), the $(i,j)$th element of the parameter matrix $\Theta$ is zero if and only if the density admits a factorization as in \cref{eq:factor}. \cite{YDS22} show that score matching can be adapted to perform model selection and inference for the interaction parameter $\Theta$ in such a setting.

The H\"usler--Reiss family has been the focus of many recent papers on high-dimensional modeling and inference for tail dependence, especially in relation to extremal graphical models \citep[see, e.g.,][]{asenova2021inference,roe2021,ELV21eglearn,HES22,LO23,RCG23}. In fact, it is the only pairwise interaction family of multivariate Pareto distributions that can be found in the current literature, despite the fact that such families are naturally related to the density factorization property which underlies extremal graphical models. A natural question is whether there exists another such family. The main result of this note answers that question.

\begin{thm}
\label{thm}
	Let $\cP$ be a family of absolutely continuous multivariate Pareto distributions in dimension $d \geq 3$ that forms a pairwise interaction model. Then $\cP$ is a subset of the H\"usler--Reiss family.
\end{thm}

\Cref{thm} justifies the focus on the H\"usler--Reiss family in the recent literature on graphical extremes. It would be tempting to extend some of the existing models in the aforementioned papers into more complicated structures while retaining the practicality of estimating pairwise interaction models, but this is in fact not possible.

In particular, in order to develop an efficient score matching algorithm, \cite{LO23} introduce the class of functions
\begin{equation} \label{eq:generalHR}
	\FgHR := \bigg\{ \by \mapsto \fgHR_{\mu, \Theta}(\by) := \frac{1}{Z(\mu, \Theta)} \exp\Big\{ -\mu^\top (\log\by) - (\log\by)^\top \Theta (\log\by) \Big\}: \mu \in \R^d, \Theta \in \cS_{\bone}^{d \times d} \bigg\}
\end{equation}
as surrogates of the densities of H\"usler--Reiss distributions. As the authors rightfully point out, $\FgHR$ strictly generalizes the class of H\"usler--Reiss densities; some of the functions therein are not even integrable, hence cannot be normalized to densities. While there are functions in $\FgHR$ which are in between, i.e., integrable but not H\"usler--Reiss densities, \cref{thm} guarantees that none of them corresponds to a multivariate Pareto distribution. In fact, a corollary of the proof of \cref{thm} is a full characterization of the functions $\fgHR_{\mu, \Theta} \in \FgHR$ based on the values of $\mu$ and $\Theta$; see \cref{lemm:FgHR} below.

\Cref{thm} is a direct consequence of the following two results, which are presented here because they are of independent interest. Their proofs are the object of \cref{sec:proofs}.

\begin{lemm}
\label{lemm:F}
	Let $\cP$ be a family of absolutely continuous multivariate Pareto distributions in dimension $d \geq 3$ that forms a pairwise interaction model and let $\cF$ be the corresponding class of densities. Then $\cF \subseteq \FgHR$.
\end{lemm}

\begin{lemm}
\label{lemm:FgHR}
	The functions $\fgHR_{\mu, \Theta}$ in $\FgHR$ can be categorized as follows.
	\begin{itemize}
		\item[(i)] If $\Theta \in \cS_{\bone, +}^{d \times d}$ and $\mu = \muHR(\Theta)$, then $\fgHR_{\mu, \Theta}$ is the density of a H\"usler--Reiss distribution.
		\item[(ii)] If $\Theta \in \cS_{\bone, +}^{d \times d}$, $\mu \neq \muHR(\Theta)$ and $\mu^\top \bone > d$, then $\fgHR_{\mu, \Theta}$ is integrable on $\cL$ but is not the density of a multivariate Pareto distribution.
		\item[(iii)] If either $\Theta \notin \cS_{\bone, +}^{d \times d}$ or $\mu^\top \bone \leq d$, then $\fgHR_{\mu, \Theta}$ is not integrable on $\cL$.
	\end{itemize}
\end{lemm}

\begin{rem}
	Part $(i)$ in \cref{lemm:FgHR} was already stated and proved by \cite{LO23}, and is in fact the main justification for working with $\FgHR$. The contribution here is in giving a complete picture of the functions in this class. In particular, there are no multivariate Pareto densities in $\FgHR$ other than H\"usler--Reiss densities.
\end{rem}

\section{Proofs}
\label{sec:proofs}

Throughout the proofs, for a vector $\by \in \R^d$, we generically define $y_i$ as the $i$th entry of $\by$ and $\by_{\setminus i} \in \R^{d-1}$ as the subvector obtained by removing its $i$th entry. The same conventions are used for indexing the rows and columns of a matrix.

\subsection{\texorpdfstring{Proof of \cref{lemm:F}}{Proof of the first lemma}}

By assumption, the density class $\cF$ is defined as in \cref{eq:pimodel}. Using the assumed properties (MP2) and (MP3) of multivariate Pareto distributions as well as the property (PI2) of pairwise interaction models, we shall establish all the required properties that will ensure that each density in $\cF$ is an element of $\FgHR$. Specifically, it will be shown that the dimension $q$ of the marginal parameters $\mu_i$ can be taken as 1, that the functions $S_i$ and $T_i$ must be logarithmic, and that moreover the parameter matrix $\Theta$ must (or rather, can be assumed to) have zero row and column sums.

\subsubsection{The functions \texorpdfstring{$T_i$}{Ti} must be logarithmic and \texorpdfstring{$\Theta$}{Theta} can be assumed to have zero row sums}

The required homogeneity property (MP2) implies that for any $\by \in \cL$ and $t>1$,
\begin{equation} \label{eq:difference}
		\sum_{i=1}^d \mu_i^\top (S_i(ty_i) -  S_i(y_i)) + \sum_{i=1}^d \sum_{j=1}^d \Theta_{ij} (T_i(ty_i) T_j(ty_j) - T_i(y_i) T_j(y_j)) = (d+1) \log t.
\end{equation}

Since for every pair $(i, j)$, $i \neq j$, $\Theta_{ij}$ can be non-zero, the above implies that for every such pair, any $\by \in \cL$ and any $t>1$,
\begin{equation} \label{eq:TiTj}
	T_i(ty_i) T_j(ty_j) - T_i(y_i) T_j(y_j) = a_{ij}(\by_{\setminus j}, t) + b_{ij}(\by_{\setminus i}, t),
\end{equation}
for some functions $a_{ij}$ and $b_{ij}$ not depending on $y_j$ and on $y_i$, respectively. Consider now the following auxiliary result.

\begin{lemm}
\label{lemm:xipsi}
	Let $\xi, \psi: (0, \infty) \to \R$ be non-constant functions such that for any $x, y \in (0, \infty)$ and $t>1$,
	\[
		\xi(tx) \psi(ty) - \xi(x) \psi(y) = \alpha(x, t) + \beta(y, t),
	\]
	for some functions $\alpha$ and $\beta$. Then there exist functions $\delta_\xi$ and $\delta_\psi$ such that for all positive $x_1$, $x_2$, $y_1$ and $y_2$ and all $t>0$,
	\[
		\xi(tx_1) - \xi(tx_2) = \delta_\xi(t) (\xi(x_1) - \xi(x_2)), \quad \psi(ty_1) - \psi(ty_2) = \delta_\psi(t) (\psi(y_1) - \psi(y_2))
	\]
	and such that $\delta_\xi(t)$ and $\delta_\psi(t)$ are non-zero, for any $t>0$.
\end{lemm}

\begin{proof}
	
	Let $x_1, x_2, y \in (0, \infty)$ and $t>1$. Applying our assumption to the points $(x_1, y)$ and $(x_2, y)$, find that
	\begin{equation} \label{eq:finitediff1}
		\psi(ty) (\xi(tx_1) - \xi(tx_2)) - \psi(y) (\xi(x_1) - \xi(x_2)) = \alpha(x_1, t) - \alpha(x_2, t).
	\end{equation}
	At this point, note that $\xi(x_1) = \xi(x_2)$ if and only if $\xi(tx_1) = \xi(tx_2)$. Indeed, if one of these equalities holds but not the other, \cref{eq:finitediff1} contradicts the assumption that $\psi$ is not constant. Now supposing that $\xi(x_1) \neq \xi(x_2)$, which is possible since $\xi$ is not constant, \cref{eq:finitediff1} is equivalent to
	\[
		\psi(ty) - \psi(y) \frac{\xi(x_1) - \xi(x_2)}{\xi(tx_1) - \xi(tx_2)} = \frac{\alpha(x_1, t) - \alpha(x_2, t)}{\xi(tx_1) - \xi(tx_2)}.
	\]
	Applying the same reasoning with any other pair of points $x_3, x_4$ such that $\xi(x_3) \neq \xi(x_4)$ yields
	\[
		\psi(ty) - \psi(y) \frac{\xi(x_3) - \xi(x_4)}{\xi(tx_3) - \xi(tx_4)} = \frac{\alpha(x_3, t) - \alpha(x_4, t)}{\xi(tx_3) - \xi(tx_4)}.
	\]
	Subtracting the latter equation from the former,
	\[
		\psi(y) \bigg( \frac{\xi(x_1) - \xi(x_2)}{\xi(tx_1) - \xi(tx_2)} - \frac{\xi(x_3) - \xi(x_4)}{\xi(tx_3) - \xi(tx_4)} \bigg) = \frac{\alpha(x_1, t) - \alpha(x_2, t)}{\xi(tx_1) - \xi(tx_2)} - \frac{\alpha(x_3, t) - \alpha(x_4, t)}{\xi(tx_3) - \xi(tx_4)},
	\]
	which is constant in $y$. However, $\psi$ was assumed non-constant. Deduce that the difference between parentheses has to be zero, hence for any (fixed) $t>1$, among all pairs $x_1, x_2$ such that $\xi(x_1) \neq \xi(x_2)$, the ratio $(\xi(tx_1) - \xi(tx_2))/(\xi(x_1) - \xi(x_2))$ is constant, say equal to $\delta_\xi(t)$.
	
	To summarize, we have shown that
	\[
		\xi(tx_1) - \xi(tx_2) = \delta_\xi(t) (\xi(x_1) - \xi(x_2))
	\]
	holds for every $t>1$ and $x_1$, $x_2$ such that $\xi(x_1) \neq \xi(x_2)$. By extension, it holds for every positive $x_1$ and $x_2$, since $\xi(x_1) = \xi(x_2)$ makes both sides vanish. Finally, the same clearly holds for $t \in(0, 1]$ if $\delta_\xi(t)$ is defined as $1$ for $t=1$, and as $\delta_\xi(t^{-1})^{-1}$ for $t<1$.
	
	This is the desired result for the function $\xi$. By symmetry, the same holds for $\psi$.
\end{proof}

By \cref{lemm:xipsi}, there exist functions $\delta_i$, $i \in \{1, \dots, d\}$, such that for every positive $y_i^{(1)}$, $y_i^{(2)}$ and $t$,
\begin{equation} \label{eq:cnsqxipsi}
	T_i(ty_i^{(1)}) - T_i(ty_i^{(2)}) = \delta_i(t) \big( T_i(y_i^{(1)}) - T_i(y_i^{(2)}) \big).
\end{equation}

Now, let $y_i^{(1)}$, $y_i^{(2)}$, $y_j^{(1)}$ and $y_j^{(2)}$ be such that $T_i(y_i^{(1)}) \neq T_i(y_i^{(2)})$ and $T_j(y_j^{(1)}) \neq T_j(y_j^{(2)})$. Define four vectors $\by^{(1)}, \dots, \by^{(4)}$, the $i$th and $j$th entries of which are $(y_i^{(1)}, y_j^{(1)})$, $(y_i^{(1)}, y_j^{(2)})$, $(y_i^{(2)}, y_j^{(1)})$ and $(y_i^{(2)}, y_j^{(2)})$, respectively, and which agree with each other in the other $d-2$ entries.

Applying \cref{eq:TiTj} to those four vectors, followed by \cref{eq:cnsqxipsi}, we have
\begin{align*}
	0 &= a_{ij}(\by_{\setminus j}^{(1)}, t) + b_{ij}(\by_{\setminus i}^{(1)}, t) - a_{ij}(\by_{\setminus j}^{(2)}, t) - b_{ij}(\by_{\setminus i}^{(2)}, t) - a_{ij}(\by_{\setminus j}^{(3)}, t) - b_{ij}(\by_{\setminus i}^{(3)}, t) + a_{ij}(\by_{\setminus j}^{(4)}, t) + b_{ij}(\by_{\setminus i}^{(4)}, t)
	\\
	&= (T_i(ty_i^{(1)}) - T_i(ty_i^{(2)})) \times (T_j(ty_j^{(1)}) - T_j(ty_j^{(2)})) - (T_i(y_i^{(1)}) - T_i(y_i^{(2)})) \times (T_j(y_j^{(1)}) - T_j(y_j^{(2)}))
	\\
	&= (\delta_i(t) \delta_j(t) - 1) \times (T_i(y_i^{(1)}) - T_i(y_i^{(2)})) \times (T_j(y_j^{(1)}) - T_j(y_j^{(2)}))
\end{align*}
for any $t>1$. By assumption, the last two terms in the product are non-zero. Deduce that $\delta_i(t) \delta_j(t) = 1$. However, for a third index $k \notin \{i, j\}$, we may apply the same logic to find that similarly, $\delta_i(t) \delta_k(t) = \delta_j(t) \delta_k(t) = 1$. This is only possible if $\delta_i(t) = \delta_j(t) = \delta_k(t) = 1$ for every $t>1$, and by extension for every $t>0$, recalling that $\delta_i(t)$ is defined as $1$ for $t=1$ and as $\delta_i(t^{-1})^{-1}$ for $t<1$. The same argument applies to every triple $(i, j, k)$, so that \cref{eq:cnsqxipsi} can be rewritten as
\[
	T_i(ty_i^{(1)}) - T_i(ty_i^{(2)}) = T_i(y_i^{(1)}) - T_i(y_i^{(2)}),
\]
which now holds for every index $i \in \{1, \dots, d\}$ and positive $y_i^{(1)}$, $y_i^{(2)}$ and $t$. Equivalently, for every $t>0$ the function $y \mapsto T_i(ty) - T_i(y)$ is constant. We now apply the following.

\begin{lemm}
\label{lemm:xi}
	Let $\xi: (0, \infty) \to \R$ be a measurable function such that for any $t>0$, the function
	\[
		x \mapsto \xi(tx) - \xi(x)
	\]
	is constant over $x>0$. Then $\xi(x) = c \log x + \xi(1)$, for some $c \in \R$.
\end{lemm}

\begin{proof}
	
	For every positive $x$ and $t$, $\xi$ satisfies $\xi(tx) - \xi(x) = \xi(t) - \xi(1)$, that is
	\[
		\xi(tx) - \xi(1) = \xi(x) + \xi(t) - 2\xi(1).
	\]
	Equivalently, in terms of the function $\varphi:= \xi \circ \exp$,
	\[
		\varphi(u+v) - \varphi(0) = \varphi(u) + \varphi(v) - 2\varphi(0), \quad u, v \in \R.
	\]
	That is, $\varphi - \varphi(0)$ satisfies Cauchy's functional equation, the only measurable solutions to which are additive functions of the form $u \mapsto cu$, for some $c \in \R$ \citep[see Theorem 1.1.8 of][and the references therein]{BGT1987,K47}. Thus, $\xi(x) = \varphi(\log x) = c\log x + \xi(1)$.
\end{proof}

Applying \cref{lemm:xi}, deduce that $T_i$ is a logarithmic function of the form $T_i(y) = c_i \log y + T_i(1)$. Note however that the values of $c_i$ and $T_i(1)$ can be absorbed into the parameters $\mu_i$ and $\Theta_{ij}$, $j \in \{1, \dots, d\}$, and into the normalizing constant, and are thus not important in characterizing the possible distributions in $\cP$. We may therefore assume that all the functions $T_1, \dots, T_d$ are equal to the logarithm function.

With our current formulation of the distributions in $\cP$, the diagonal elements of $\Theta$ are not identifiable. Indeed, their value can be changed arbitrarily by adding a $(q+1)$th dimension to each $\mu_i$ and letting $S_i(y_i)_{q+1} = (\log y_i)^2$. Therefore, we shall assume without loss of generality that $\Theta_{ii} = -\sum_{j \neq i} \Theta_{ij}$, so that the row sums of $\Theta$ (and by symmetry, its column sums) are all equal to zero.

\subsubsection{The parameters \texorpdfstring{$\mu_i$}{mui} can be assumed scalar and the functions \texorpdfstring{$S_i$}{Si} must be logarithmic}

Replacing the functions $T_i$ by logarithms in \cref{eq:difference} and using the assumption that the row and columns sums of $\Theta$ are zero, we find
\begin{align*}
	(d+1) \log t &= \sum_{i=1}^d \mu_i^\top (S_i(ty_i) -  S_i(y_i)) + \sum_{i=1}^d \sum_{j=1}^d \Theta_{ij} (\log y_i + \log y_j + \log t) \log t = \sum_{i=1}^d \mu_i^\top (S_i(ty_i) -  S_i(y_i)).
\end{align*}

Similarly to what was argued about the functions $T_i$, deduce that for all $y_i$ and $t>0$,
\[
	\mu_i^\top (S_i(ty_i) - S_i(y_i)) = \mu_i^\top (S_i(t) - S_i(1)),
\]
which by \cref{lemm:xi} means that $\mu_i^\top S_i$ can be chosen to be simply a logarithm, up to scaling. Thus, we may assume without loss of generality that $q=1$ and that the real-valued functions $S_i$ are logarithms.

To summarize, we have established that all the densities in $\cF$ must be of the form
\[
	f_{\mu, \Theta}(\by) = \frac{1}{Z(\mu, \Theta)} \exp\Big\{ -\sum_{i=1}^d \mu_i (\log y_i) - \sum_{i=1}^d \sum_{j=1}^d \Theta_{ij} (\log y_i)(\log y_j) \Big\}
\]
with $\Theta$ symmetric with zero row (and column) sums. This concludes the proof.
\hfill$\square$

\subsection{\texorpdfstring{Proof of \cref{lemm:FgHR}}{Proof of the second lemma}}

As mentioned after the statement of the result, $(i)$ is already obtained by \cite{LO23}, so only $(ii)$ and $(iii)$ shall be proved here.

Let $\fgHR_{\mu, \Theta} \in \FgHR$. It will first be shown that for $\fgHR_{\mu, \Theta}$ to be integrable on $\cL$, it is necessary for $\Theta$ to be a H\"usler--Reiss precision matrix, i.e., an element of $\cS_{\bone, +}^{d \times d}$, and for $\mu$ to satisfy $\mu^\top \bone > d$, establishing $(iii)$. Finally, it will be shown that for such a given matrix $\Theta$, for $\fgHR_{\mu, \Theta}$ to be a multivariate Pareto density, it is necessary for $\mu$ to have the specific form $\muHR(\Theta)$ in which case $\fgHR_{\mu, \Theta}$ is a H\"usler--Reiss density, thus establishing $(ii)$.

\subsubsection{If \texorpdfstring{$\Theta \notin \cS_{\bone, +}^{d \times d}$}{Theta is not a H\"usler--Reiss precision matrix}, then \texorpdfstring{$\fgHR_{\mu, \Theta}$}{fgHR} is not integrable}

For any index $k$, by the change of variable $\bx = \log\by$, we find that
\begin{align}
	Z(\mu, \Theta) \int_{\{\by \in \cL: y_k > 1\}} \fgHR_{\mu, \Theta}(\by) d\by &= \int_{\{\by \in \cL: y_k > 1\}} \exp\Big\{ -\sum_{i=1}^d \mu_i \log(y_i) - \sum_{i=1}^d \sum_{j=1}^d \Theta_{ij} (\log y_i)(\log y_j) \Big\} d\by \notag
	\\
	&= \int_{\{\bx \in \R^d: x_k > 0\}} \exp\Big\{ -(\mu - \bone)^\top \bx - \bx^\top \Theta \bx \Big\} d\bx. \label{eq:PYk>1}
\end{align}
Decomposing $\bx$ into $\bx_{\setminus k}$ and $x_k$, we may write $(\mu - \bone)^\top \bx$ as $(\mu - \bone)_{\setminus k}^\top \bx_{\setminus k} + (\mu_k - 1) x_k$, and $\bx^\top \Theta \bx$ as
\[
	\bx_{\setminus k}^\top \Theta^{(k)} \bx_{\setminus k} + 2x_k \Theta_{k, \setminus k} \bx_{\setminus k} + \Theta_{kk} x_k^2,
\]
where $\Theta^{(k)} := \Theta_{\setminus k, \setminus k}$. We may then rewrite \cref{eq:PYk>1} as
\begin{align*}
	&\int_0^\infty \int_{\R^{d-1}} \exp\Big\{ -\bx_{\setminus k}^\top \Theta^{(k)} \bx_{\setminus k} - 2x_k \Theta_{k, \setminus k} \bx_{\setminus k} - (\mu - \bone)_{\setminus k}^\top \bx_{\setminus k} \Big\} d\bx_{\setminus k} \times \exp\Big\{ -(\mu_k - 1) x_k - \Theta_{kk} x_k^2 \Big\} dx_k.
\end{align*}
The inner integral is a Gaussian type integral. It is straightforward to show that it is finite if and only if $\Theta^{(k)}$ is positive definite, using a spectral decomposition of that matrix. Deduce that for $\fgHR_{\mu, \Theta}$ to be integrable, all the matrices $\Theta^{(k)}$ must be positive definite (hence, of full rank $d-1$). This implies that $\Theta$ must also be of rank $d-1$. Moreover, $\Theta$ has to have only non-negative eigenvalues. Indeed, suppose it doesn't. Then there is an $\bx \in \R^d$ such that $\bx^\top \Theta \bx < 0$. However, since $\Theta\bone = 0$, it is also true that $0 > (\bx - x_k \bone)^\top \Theta (\bx - x_k \bone) = (\bx - x_k \bone)_{\setminus k}^\top \Theta^{(k)} (\bx - x_k \bone)_{\setminus k}$, contradicting the positive definiteness of $\Theta^{(k)}$.

It is therefore necessary for the integrability of $\fgHR_{\mu, \Theta}$ on $\cL$ that $\Theta \in \cS_{\bone, +}^{d \times d}$. For the remainder of the proof, we shall assume that this is the case.

\subsubsection{If \texorpdfstring{$\mu^\top \bone \leq d$}{the elements of mu sum to at most d}, then \texorpdfstring{$\fgHR_{\mu, \Theta}$}{fgHR} is not integrable}

Now that we assume the matrices $\Theta^{(k)}$ to be invertible, we may obtain a more precise expression for \cref{eq:PYk>1}. By tedious but elementary computations involving completion of the quadratic form $\bx_{\setminus k}^\top \Theta^{(k)} \bx_{\setminus k}$, we may rewrite the argument of the exponential in \cref{eq:PYk>1} as
\[
	- \big( \bx_{\setminus k} - \beta^{(k)}(x_k) \big)^\top \Theta^{(k)} \big( \bx_{\setminus k} - \beta^{(k)}(x_k) \big) + \beta^{(k)}(x_k)^\top \Theta^{(k)} \beta^{(k)}(x_k) + (1 - \mu_k) x_k - \Theta_{kk} x_k^2,
\]
where $\beta^{(k)}(x_k) := \tfrac{1}{2} \Sigma^{(k)} (\bone - \mu)_{\setminus k} + x_k \bone_{\setminus k}$ and $\Sigma^{(k)} := (\Theta^{(k)})^{-1}$. Since $\bx_{\setminus k}$ only appears in the first term, we may rewrite \cref{eq:PYk>1} as
\begin{align}
	&\int_0^\infty \bigg[ \int_{\R^{d-1}} \exp\Big\{ -\big( \bx_{\setminus k} - \beta^{(k)}(x_k) \big)^\top \Theta^{(k)} \big( \bx_{\setminus k} - \beta^{(k)}(x_k) \big) \Big\} d\bx_{\setminus k} \notag
	\\
	&\quad\quad \times \exp\Big\{ \beta^{(k)}(x_k)^\top \Theta^{(k)} \beta^{(k)}(x_k) + (1 - \mu_k) x_k - \Theta_{kk} x_k^2 \Big\} \bigg] dx_k \notag
	\\
	&\quad = \frac{(2\pi)^{(d-1)/2}}{\det(\Theta^{(k)})^{1/2}} \int_0^\infty \exp\Big\{ \beta^{(k)}(x_k)^\top \Theta^{(k)} \beta^{(k)}(x_k) + (1 - \mu_k) x_k - \Theta_{kk} x_k^2 \Big\} dx_k. \label{eq:PYk>1new}
\end{align}

Expanding the quadratic form $\beta^{(k)}(x_k)^\top \Theta^{(k)} \beta^{(k)}(x_k)$, we find that
\begin{align*}
	\beta^{(k)}(x_k)^\top \Theta^{(k)} \beta^{(k)}(x_k) + (1 - \mu_k) x_k - \Theta_{kk} x_k^2 &= -\big( (\mu - \bone)^\top \bone\big) x_k + \frac{1}{4} (\mu - \bone)_{\setminus k}^\top \Sigma^{(k)} (\mu - \bone)_{\setminus k}
	\\
	&= -(\mu^\top \bone - d)x_k + \frac{1}{4} (\mu - \bone)_{\setminus k}^\top \Sigma^{(k)} (\mu - \bone)_{\setminus k}.
\end{align*}
Therefore, \cref{eq:PYk>1new} is equal to
\begin{equation} \label{eq:PYk>1newnew}
	\frac{(2\pi)^{(d-1)/2}}{\det(\Theta^{(k)})^{1/2}} \exp\Big\{ \frac{1}{4} (\mu - \bone)_{\setminus k}^\top \Sigma^{(k)} (\mu - \bone)_{\setminus k} \Big\} \int_0^\infty e^{-(\mu^\top \bone - d)x_k} dx_k,
\end{equation}
which is finite if and only if $\mu^\top \bone > d$. This establishes that $\fgHR_{\mu, \Theta}$ is integrable on $\cL$ if and only if $\Theta \in \cS_{\bone, +}^{d \times d}$ and $\mu^\top \bone > d$, which in particular implies $(iii)$.

\subsubsection{If \texorpdfstring{$\fgHR_{\mu, \Theta}$}{fgHR} is a multivariate Pareto density, then \texorpdfstring{$\mu = \muHR(\Theta)$}{mu = muHR(Theta)}}

Now suppose that $\fgHR_{\mu, \Theta}$ is a multivariate Pareto density. By the last two sections, it follows that $\Theta \in \cS_{\bone, +}^{d \times d}$. Moreover, by the homogeneity property (MP2) and the fact that $\Theta\bone = 0$,
\[
	\log\fgHR_{\mu, \Theta}(\by) - \log\fgHR_{\mu, \Theta}(t\by) = \sum_{i=1}^d \mu_i \log t + \sum_{i=1}^d \sum_{j=1}^d \Theta_{ij} (\log y_i + \log y_j + \log t) \log t = \sum_{i=1}^d \mu_i \log t
\]
must be equal to $(d+1) \log t$. That is, $\mu^\top \bone = d+1$.

We shall now enforce the marginal standardization property (MP3). Recalling \cref{eq:PYk>1newnew}, we now have
\[
	Z(\mu, \Theta) \int_{\{\by \in \cL: y_k > 1\}} \fgHR_{\mu, \Theta}(\by) d\by = \frac{(2\pi)^{(d-1)/2}}{\det(\Theta^{(k)})^{1/2}} \exp\Big\{ \frac{1}{4} (\mu - \bone)_{\setminus k}^\top \Sigma^{(k)} (\mu - \bone)_{\setminus k} \Big\}.
\]

By Equation (23) in \cite{roe2021}, $\det(\Theta^{(k)})$ is in fact the pseudodeterminant of $\Theta$ and as such, does not depend on $k$. Hence, the marginal standardization property holds if and only if the value of $(\mu - \bone)_{\setminus k}^\top \Sigma^{(k)} (\mu - \bone)_{\setminus k}$ is the same for each $k$.

The matrices $\Sigma^{(k)}$, however, enjoy a special structure. Let us augment $\Sigma^{(k)}$ by adding a row and column of zeros in its $k$th position, forming a matrix $\tilde\Sigma^{(k)} \in \R^{d \times d}$. Then the matrices $\tilde\Sigma^{(k)}$ satisfy
\[
	\tilde\Sigma_{ij}^{(k)} = \frac{1}{2} (\Gamma_{ik} + \Gamma_{jk} - \Gamma_{ij}), \quad i, j \in \{1, \dots, d\},
\]
where $\Gamma$ is the variogram matrix associated to $\Theta$, as defined in \cref{ex:HR}. These are the same matrices $\tilde\Sigma^{(k)}$ as that appearing in Section 4.3 of \cite{EH2020}. Then, for any two indices $k$ and $\ell$,
\begin{align*}
	(\mu - \bone)_{\setminus k}^\top \Sigma^{(k)} (\mu - \bone)_{\setminus k} - (\mu - \bone)_{\setminus \ell}^\top \Sigma^{(\ell)} (\mu - \bone)_{\setminus \ell} &= (\mu - \bone)^\top \tilde\Sigma^{(k)} (\mu - \bone) - (\mu - \bone)^\top \tilde\Sigma^{(\ell)} (\mu - \bone)
	\\
	&= \frac{1}{2} \sum_{i=1}^d \sum_{j=1}^d \big( \tilde\Sigma_{ij}^{(k)} - \tilde\Sigma_{ij}^{(\ell)} \big) (\mu_i - 1)(\mu_j - 1)
	\\
	&= \frac{1}{2} \sum_{i=1}^d \sum_{j=1}^d (\Gamma_{ik} - \Gamma_{i\ell} + \Gamma_{jk} - \Gamma_{j\ell}) (\mu_i - 1)(\mu_j - 1)
	\\
	&= \sum_{i=1}^d (\Gamma_{ik} - \Gamma_{i\ell})(\mu_i - 1) \sum_{j=1}^d (\mu_j - 1)
	\\
	&= (\Gamma_{k\cdot} - \Gamma_{\ell\cdot}) (\mu - \bone),
\end{align*}
recalling the fact that $\mu^\top \bone = d+1$, or equivalently $(\mu - \bone)^\top \bone = 1$. This can only be zero for any $k \neq \ell$ if $\Gamma_{k\cdot}(\mu - \bone)$ has the same value for every $k$, i.e., $\Gamma(\mu - \bone) \in \{\gamma\bone: \gamma \in \R\}$. By invertibility of $\Gamma$, this forms a non-singular system of $d-1$ linear equations. The solution set is a one-dimensional linear subspace, only one element of which, say $\mu^* - \bone$, satisfies $(\mu^* - \bone)^\top \bone = 1$. Therefore, for any given parameter matrix $\Theta \in \cS_{\bone, +}^{d \times d}$, the parameter vector $\mu$ is uniquely determined by the multivariate Pareto structure. The resulting distribution can be none other than the H\"usler--Reiss distribution with precision matrix $\Theta$.

Note that it can be verified by simple calculations that the unique solution to the above linear system is indeed the H\"usler--Reiss distribution. Indeed, using Lemma S.5.11 of \cite{HES22}, it can be confirmed that $\Gamma(\muHR(\Theta) - \bone)$ is indeed the vector $\bone$ multiplied by the scalar $d^{-2} \bone^\top \Gamma (\tfrac{1}{2} \Theta\Gamma - I) \bone$, where $I$ denotes the identity matrix, and that $(\muHR(\Theta) - \bone)^\top \bone = 1$.
\hfill$\square$

%
%


\begin{thebibliography}{}

\bibitem[\protect\citeauthoryear{Asenova, Mazo, and Segers}{Asenova
  et~al.}{2021}]{asenova2021inference}
Asenova, S., G.~Mazo, and J.~Segers (2021).
\newblock Inference on extremal dependence in the domain of attraction of a
  structured {H}\"{u}sler-{R}eiss distribution motivated by a {M}arkov tree
  with latent variables.
\newblock {\em Extremes\/}~{\em 24\/}(3), 461--500.

\bibitem[\protect\citeauthoryear{Bingham, Goldie, and Teugels}{Bingham
  et~al.}{1987}]{BGT1987}
Bingham, N., C.~Goldie, and J.~L. Teugels (1987).
\newblock {\em {Regular Variation}}.
\newblock Cambridge University Press.

\bibitem[\protect\citeauthoryear{Brown}{Brown}{1986}]{Brown86}
Brown, L.~D. (1986).
\newblock {\em Fundamentals of statistical exponential families with
  applications in statistical decision theory}, Volume~9 of {\em Institute of
  Mathematical Statistics Lecture Notes---Monograph Series}.
\newblock Institute of Mathematical Statistics, Hayward, CA.

\bibitem[\protect\citeauthoryear{Darroch, Lauritzen, and Speed}{Darroch
  et~al.}{1980}]{DLS80}
Darroch, J.~N., S.~L. Lauritzen, and T.~P. Speed (1980).
\newblock Markov fields and log-linear interaction models for contingency
  tables.
\newblock {\em Ann. Statist.\/}~{\em 8\/}(3), 522--539.

\bibitem[\protect\citeauthoryear{Engelke and Hitz}{Engelke and
  Hitz}{2020}]{EH2020}
Engelke, S. and A.~S. Hitz (2020).
\newblock Graphical models for extremes.
\newblock {\em J. R. Stat. Soc. Ser. B. Stat. Methodol.\/}~{\em 82\/}(4),
  871--932.
\newblock With discussions.

\bibitem[\protect\citeauthoryear{Engelke, Lalancette, and Volgushev}{Engelke
  et~al.}{2022}]{ELV21eglearn}
Engelke, S., M.~Lalancette, and S.~Volgushev (2022).
\newblock Learning extremal graphical structures in high dimensions.
\newblock {\em arXiv preprint arXiv:2111.00840\/}.

\bibitem[\protect\citeauthoryear{Engelke and Volgushev}{Engelke and
  Volgushev}{2022}]{EV22}
Engelke, S. and S.~Volgushev (2022).
\newblock Structure learning for extremal tree models.
\newblock {\em J. R. Stat. Soc. Ser. B. Stat. Methodol.\/}~{\em 84\/}(5),
  2055--2087.

\bibitem[\protect\citeauthoryear{Hentschel, Engelke, and Segers}{Hentschel
  et~al.}{2022}]{HES22}
Hentschel, M., S.~Engelke, and J.~Segers (2022).
\newblock Statistical inference for {H\"u}sler--{R}eiss graphical models
  through matrix completions.
\newblock {\em arXiv preprint arXiv:2210.14292\/}.

\bibitem[\protect\citeauthoryear{H{\"u}sler and Reiss}{H{\"u}sler and
  Reiss}{1989}]{HR1989}
H{\"u}sler, J. and R.-D. Reiss (1989).
\newblock Maxima of normal random vectors: between independence and complete
  dependence.
\newblock {\em Stat. Probab. Lett.\/}~{\em 7\/}(4), 283--286.

\bibitem[\protect\citeauthoryear{Inouye, Ravikumar, and Dhillon}{Inouye
  et~al.}{2016}]{IRD16}
Inouye, D., P.~Ravikumar, and I.~Dhillon (2016).
\newblock Square root graphical models: Multivariate generalizations of
  univariate exponential families that permit positive dependencies.
\newblock In {\em International conference on machine learning}, pp.\
  2445--2453. PMLR.

\bibitem[\protect\citeauthoryear{Kestelman}{Kestelman}{1947}]{K47}
Kestelman, H. (1947).
\newblock On the functional equation $f(x+y)=f(x)+f(y)$.
\newblock {\em Fundamenta Mathematicae\/}~{\em 1\/}(34), 144--147.

\bibitem[\protect\citeauthoryear{Kiriliouk, Rootz{\'e}n, Segers, and
  Wadsworth}{Kiriliouk et~al.}{2019}]{KRSW19}
Kiriliouk, A., H.~Rootz{\'e}n, J.~Segers, and J.~L. Wadsworth (2019).
\newblock Peaks over thresholds modeling with multivariate generalized pareto
  distributions.
\newblock {\em Technometrics\/}~{\em 61\/}(1), 123--135.

\bibitem[\protect\citeauthoryear{Klein, Orellana, Brincat, Miller, and
  Kass}{Klein et~al.}{2020}]{KOBMK20}
Klein, N., J.~Orellana, S.~L. Brincat, E.~K. Miller, and R.~E. Kass (2020).
\newblock Torus graphs for multivariate phase coupling analysis.
\newblock {\em Ann. Appl. Stat.\/}~{\em 14\/}(2), 635--660.

\bibitem[\protect\citeauthoryear{Lederer and Oesting}{Lederer and
  Oesting}{2023}]{LO23}
Lederer, J. and M.~Oesting (2023).
\newblock Extremes in high dimensions: Methods and scalable algorithms.
\newblock {\em arXiv preprint arXiv:2303.04258\/}.

\bibitem[\protect\citeauthoryear{Lin, Drton, and Shojaie}{Lin
  et~al.}{2016}]{LDS16}
Lin, L., M.~Drton, and A.~Shojaie (2016).
\newblock Estimation of high-dimensional graphical models using regularized
  score matching.
\newblock {\em Electron. J. Stat.\/}~{\em 10\/}(1), 806--854.

\bibitem[\protect\citeauthoryear{Liu, Lafferty, and Wasserman}{Liu
  et~al.}{2009}]{LLW09}
Liu, H., J.~Lafferty, and L.~Wasserman (2009).
\newblock The nonparanormal: semiparametric estimation of high dimensional
  undirected graphs.
\newblock {\em J. Mach. Learn. Res.\/}~{\em 10}, 2295--2328.

\bibitem[\protect\citeauthoryear{Maathuis, Drton, Lauritzen, and
  Wainwright}{Maathuis et~al.}{2019}]{MDLW18}
Maathuis, M., M.~Drton, S.~Lauritzen, and M.~Wainwright (Eds.) (2019).
\newblock {\em Handbook of graphical models}.
\newblock Chapman \& Hall/CRC Handbooks of Modern Statistical Methods. CRC
  Press, Boca Raton, FL.

\bibitem[\protect\citeauthoryear{Rootz{\'e}n, Segers, and
  Wadsworth}{Rootz{\'e}n et~al.}{2018}]{RSW18POT}
Rootz{\'e}n, H., J.~Segers, and J.~L. Wadsworth (2018).
\newblock Multivariate peaks over thresholds models.
\newblock {\em Extremes\/}~{\em 21\/}(1), 115--145.

\bibitem[\protect\citeauthoryear{Rootz{\'e}n and Tajvidi}{Rootz{\'e}n and
  Tajvidi}{2006}]{roo2006}
Rootz{\'e}n, H. and N.~Tajvidi (2006).
\newblock Multivariate generalized {P}areto distributions.
\newblock {\em Bernoulli\/}~{\em 12}, 917--930.

\bibitem[\protect\citeauthoryear{R{\"o}ttger, Coons, and Grosdos}{R{\"o}ttger
  et~al.}{2023}]{RCG23}
R{\"o}ttger, F., J.~I. Coons, and A.~Grosdos (2023).
\newblock Parametric and nonparametric symmetries in graphical models for
  extremes.
\newblock {\em arXiv preprint arXiv:2306.00703\/}.

\bibitem[\protect\citeauthoryear{R{\"o}ttger, Engelke, and
  Zwiernik}{R{\"o}ttger et~al.}{2023}]{roe2021}
R{\"o}ttger, F., S.~Engelke, and P.~Zwiernik (2023).
\newblock Total positivity in multivariate extremes.
\newblock {\em Ann. Statist.\/}, to appear.

\bibitem[\protect\citeauthoryear{Yang, Ravikumar, Allen, and Liu}{Yang
  et~al.}{2015}]{YRAL15}
Yang, E., P.~Ravikumar, G.~I. Allen, and Z.~Liu (2015).
\newblock Graphical models via univariate exponential family distributions.
\newblock {\em J. Mach. Learn. Res.\/}~{\em 16}, 3813--3847.

\bibitem[\protect\citeauthoryear{Yu, Drton, and Shojaie}{Yu
  et~al.}{2022}]{YDS22}
Yu, S., M.~Drton, and A.~Shojaie (2022).
\newblock Generalized score matching for general domains.
\newblock {\em Inf. Inference\/}~{\em 11\/}(2), 739--780.

\end{thebibliography}

\end{document}